\documentclass{article}

\usepackage[utf8]{inputenc}
\usepackage{bbold}

\newcommand\numeq[1]%
  {\stackrel{\scriptscriptstyle(\mkern-1.5mu#1\mkern-1.5mu)}{=}}

\usepackage{url}
\usepackage{enumitem}
\usepackage{xcolor}
\usepackage{chngpage}
\usepackage{amsmath}
\usepackage{amsfonts}
\usepackage{amsthm}
\usepackage{esvect}
\usepackage{todonotes}
\usepackage{ amssymb }
\usepackage{mathtools}
\usepackage{subfiles}
\usepackage{comment}
\usepackage{diagbox}
\usepackage[maxnames=10]{biblatex}
\usepackage{subfig}
\usepackage{graphicx}
\usepackage{tikz}

\definecolor{darkcerulean}{rgb}{0.03, 0.27, 0.49}
\usepackage[colorlinks=true,linkcolor=darkcerulean,citecolor=brown]{hyperref}
\DeclarePairedDelimiter\floor{\lfloor}{\rfloor}

\numberwithin{equation}{section}
\newcolumntype{L}{>{$}l<{$}}
\newcolumntype{C}{>{$}c<{$}}
\newcommand\n{n}
\newcommand{\suchthat}{\;\ifnum\currentgrouptype=16 \middle\fi|\;}
\newtheorem{theorem}{Theorem}[section]
\newtheorem{conjecture}[theorem]{Conjecture}
\newtheorem{corollary}[theorem]{Corollary}
\newtheorem{claim}[theorem]{Claim}
\newtheorem{note}[theorem]{Note}
\newtheorem{definition}[theorem]{Definition}
\newtheorem{question}[theorem]{Question}
\newtheorem{lemma}[theorem]{Lemma}
\newtheorem{proposition}[theorem]{Proposition}

\title{Hyperpaths}
\author{{Amir Dahari}\thanks{School of Computer Science and Engineering, Hebrew University, Jerusalem 91904, Israel. e-mail: amir.dahari@mail.huji.ac.il.} \and {Nathan Linial}\thanks{School of Computer Science and Engineering, Hebrew University, Jerusalem 91904, Israel. e-mail: nati@cs.huji.ac.il. Supported by BSF US-Israel grant 2018313 "Between topology and combinatorics"}}
\date{}

\begin{document}

\maketitle
\begin{abstract}
Hypertrees are high-dimensional counterparts of graph theoretic trees. They have attracted a great deal of attention by various investigators. Here
we introduce and study {\em Hyperpaths} - a particular class of hypertrees which are high dimensional analogs of paths in graph theory. A $d$-dimensional hyperpath is a $d$-dimensional hypertree in which every $(d-1)$-dimensional face is contained in at most $(d+1)$ faces of dimension $d$. We introduce a possibly infinite family of hyperpaths for every dimension, and investigate its properties in greater depth for dimension $d=2$.
\end{abstract}

\section{Introduction}

{\em Hypertrees} were defined in 1983 by Kalai \cite{kalai1983enumeration}. An $n$-vertex $d$-dimensional hypertree $X=(V,E)$ is a $\mathbb{Q}$-acyclic $d$-dimensional simplicial complex with a full $(d-1)$-dimensional skeleton and $\binom{n-1}{d}$ faces of dimension $d$. Note that when $d=1$ this coincides with the usual notion of a tree in graph theory. Also note that a hypertree is completely specified by its list of $d$-dimensional faces. There is already a sizable literature, e.g., \cite{linial2019enumeration, linial2016phase} dealing with hypertrees, but many basic questions in this area are still open. Also, in order to develop an intuition for these complexes it is desirable to have a large supply of different constructions. Many investigations in this area are done with an eye to the one-dimensional situation. Arguably the simplest-to-describe ($1$-dimensional) trees are stars and paths. These two families of trees are also the two extreme examples with respect to various natural graph properties and parameters such as the tree's diameter. Hyperstars are very easy to describe in any dimension $d$. Namely, we pick a vertex $v\in V$ and put a $d$-face $\sigma$ in $E$ iff $v\in\sigma$. On the other hand, it is much less obvious how to define $d$-dimensional paths. A one-dimensional path is a tree in which every vertex has degree at most $2$. Working by analogy we can define a $d$-dimensional hyperpath as a $d$-dimensional hypertree in which every $(d-1)$-dimensional face is contained in no more than $(d+1)$ faces of dimension $d$. We include a summary of the main results presented in this paper:
\begin{enumerate}
    \item We introduce an infinite family of $d$-dimensional algebraically-defined simplicial complexes. In dimension $d=2$ we analyzed fairly large (up to $n\sim1400$) such complexes most of which turned out to be $2$-dimensional hyperpaths. To this end we devised a new fast algorithm that determines whether a matrix with circulant blocks is invertible. \label{computational_results}
    \item Negative results: We showed that infinitely many of the $2$-dimensional complexes discussed in Item \ref{computational_results} are {\em not} $\mathbb{Q}$-acyclic.
    \item We develop several approaches for proving positive results and finding an infinite family of $2$-dimensional hyperpaths.
\end{enumerate}

\begin{note}
The necessary background in simplicial combinatorics and in number theory are introduced in section \ref{section:preliminaries}.
\end{note}

\begin{definition}\label{n_c_definition}
Let $\mathbb{F}_{\n}$ be the field of prime order $n$. For $c\in {\mathbb{F}}_{\n}$ and $d\ge 1$ an integer, we define the complex $X=X_{d,n,c}$ on vertex set $\mathbb{F}_{\n}$. It has a full $(d-1)$-dimensional skeleton, and $\{x_0, x_1,\ldots, x_d\}$ is a $d$-face in $X$ iff $cx_d+\sum_{j=0}^{d-1} x_j\equiv0 \mod{n}$. \footnote{Throughout this paper, unless stated otherwise, given a prime $n$, all arithmetic equations are $\bmod n$, and we often replace the congruence relation $\equiv$ by an equality sign when no confusion is possible.}
\end{definition}

Ours is by no means the only sensible definition of a hyperpath. An alternative approach is described in \cite{mathew2015boundaries}. That paper starts from the observation that a ($1$-dimensional) path is characterised as a tree that can be made a spanning cycle by adding a single edge. In this view they define a Hamiltonian $d$-cycle as a simple $d$-cycle of size $\binom{n-1}{d}+1$. A Hamiltonian $d$-dimensional hyperpath is defined as a $d$-dimensional hypertree which can be made a Hamiltonian $d$-cycle by adding a single $d$-face. Other possibilities suggest themselves. For example, when an edge is added to a tree a single cycle is created. One may wonder how the length of this cycle is distributed when the added edge is chosen randomly. A path is characterized as the tree for which the average of this length is maximized. Similar notions clearly make sense also for $d>1$. These various definitions coincide when $d=1$ but disagree for $d>1$. It would be interesting to understand the relations between these different definitions. Note that sum complexes \cite{linial2010sum} as well as certain hypertrees from \cite{linial2019enumeration} are hyperpaths according to our definition.

\subsection*{Running Example}
We repeatedly return throughout the paper to the example corresponding to $d=2,n=13,c=5$. Some of the $2$-faces in $X_{2,13,5}$ are $\{0,1,5\},\{2,3,12\},\{2,9,3\}$ since $0+1+5\cdot5= 2+3+5\cdot12=2+9+5\cdot3=0$. As the next claim shows the number of $2$-faces in $X_{2,13,5}$ is $\binom{12}{2}=66$ (out of the total of $\binom{13}{3}=286\ 2$-faces).
\vspace{2mm}
\hrule
\vspace{2mm}
Given the vertex set $\{x_0, x_1,\ldots, x_d\}$ of a $d$-face as in definition \ref{n_c_definition}, and if $c\neq1$, the coordinate that is multiplied by $c$ is uniquely defined. For if $cx_d+\sum_{j=0}^{d-1} x_j=cx_0+\sum_{j=1}^{d} x_j=0$, then $(c-1)\cdot x_d=(c-1)\cdot x_0$. This is impossible, since we are assuming that all $x_i$ are distinct, and $c\neq 1$.

\begin{claim}\label{right_number}
For an integer $d\ge 1$, a prime $\n$ and $c\in\mathbb{F}_n$, if $c\neq{-d,1}$ then $X=X_{d,n,c}$ has exactly $\binom{n-1}{d}$ $d$-faces. If $c=1$ then $X=X_{d,n,1}$ has exactly $\frac{1}{d+1}\binom{n-1}{d}$ $d$-faces. 
\end{claim}

\begin{proof}

By induction on $d$. Let us start with $d=1$. If $c\neq0$, then for every $x\neq0$, there is a unique $y\neq0$ s.t.\ $x+c\cdot y=0$. Also, $x\neq y$, since by assumption $c\neq-1(=-d)$. This yields $\binom{n-1}{1}=n-1$ edges, unless $c=1$ in which case every edge is counted twice with a total of $\frac{1}{2}\cdot\binom{n-1}{1}$ different edges. When $c=0$, the complex has $n-1$ edges, namely, $\{0,y\}$ for all $y\neq0$. 

We proceed to deal with $d>1$, 
\begin{itemize}
    \item If $c=0$, then $\{x_0,\dots,x_{d-1},y\}$ is a $d$-face iff $\sum_{i=0}^{i=d-2}x_i+1\cdot x_{d-1}=0$. By the induction hypothesis with $c=1$ and dimension $d-1$ there are exactly $\frac{1}{d}\binom{n-1}{d-1}$ such different choices of $\{x_0,\dots,x_{d-1}\}$. For every such choice of $\{x_0,\dots,x_{d-1}\}$ there are $n-d$ choices for $y$, namely, any value not in $\{x_0,\dots,x_{d-1}\}$, yielding a total of
    $$\frac{1}{d}\binom{n-1}{d-1}\cdot(n-d)=\binom{n-1}{d}$$ distinct $d$-faces.
    \item If $c\neq0$, for each of the $\binom{n}{d}$ $(d-1)$-faces $\{x_0,\dots,x_{d-1}\}$ there is a unique $y$ satisfying $\sum_{i=0}^{d-1}x_i+c\cdot y=0$. This gives a $d-$face, unless $y\in \{x_0,\dots,x_{d-1}\}$. By reordering the $x_i$ if necessary $y=x_{d-1}$, which yields $$\sum_{i=0}^{d-2}x_i+(c+1)\cdot x_{d-1}=0$$ 
Since $c\neq0,-d$, we know that $c+1\neq1,-(d-1)$ and we can apply the induction hypothesis for dimension $d-1$ to obtain $\binom{n-1}{d-1}$ such different choices of $\{x_0,\dots,x_{d-1}\}$. All told there are $$\binom{n}{d}-\binom{n-1}{d-1}=\binom{n-1}{d}$$ distinct $d$-faces in $X_{d,n,c}$. If $c=1$ $y$ has no special role and we over-count by a multiple of $d+1$, hence we get only $\frac{1}{d+1}\binom{n-1}{d}$ distinct $d$-faces. 
\end{itemize}
\end{proof}

Henceforth to simplify matters we assume $c\not\in\{0,1,-d,-1\}$.

In the $1$-dimensional case the resulting graph $G=(V,E)$ has $V=\mathbb{F}_n$ and $E=\{\{x, -\frac{x}{c}\}\suchthat x\neq 0\}$. Consequently $G$ is the union of $\frac{n-1}{o(-c^{-1})}$ circles of length $o(-c^{-1})$. We will later see that $o(c)$ plays a crucial rule in determining whether $X_{2,n,c}$ is a hypertree.

We note that if $X=X_{d,n,c}$ is a hypertree, then it is a hyperpath, since every $(d-1)$-face $\sigma$ in $X$ is contained in at most $d+1$ of its $d$-faces. Indeed, let $y\not\in\sigma$ be the vertex that is added to $\sigma$ to form a $d$-face. Then either $\sum_{i=0}^{d-1}x_i+c\cdot y=0$ or there is an index $d-1\ge j\ge 0$ such that $cx_j+\sum_{i\neq j}x_i+ y=0$.

\subsection*{Running Example}
The edge $\{1,5\}$ in the complex $X_{2,13,5}$ is included in the faces $\{5,1,4\},\{0,1,5\}$ and $\{5,3,1\}$ with the convention that the last vertex is multiplied by $c=5$. In contrast $\{1,4\}$ is included in only two faces, namely $\{1,4,12\}$ and $\{5,1,4\}$. The equality $4+4+5\cdot1=0$ yields the non-face $\{4,4,1\}$.
\vspace{2mm}
\hrule
\vspace{2mm}

Since we focus mostly on the $2$-dimensional case $d=2$, we use the shorthand $X=X_{2,\n,c}$. The boundary operator of $X$ is given by an $\binom{n}{2}\times \binom{n-1}{2}$ matrix which we denote by $A=A_{\n,c}$. Clearly $X$ is a hyperpath iff $A$ has a full column rank, and indeed our main technical question is:

\begin{figure}[htb!]
\centering
\begin{tikzpicture}
    \node (img1)  {\includegraphics[width=0.9\textwidth,height=\textheight,keepaspectratio]{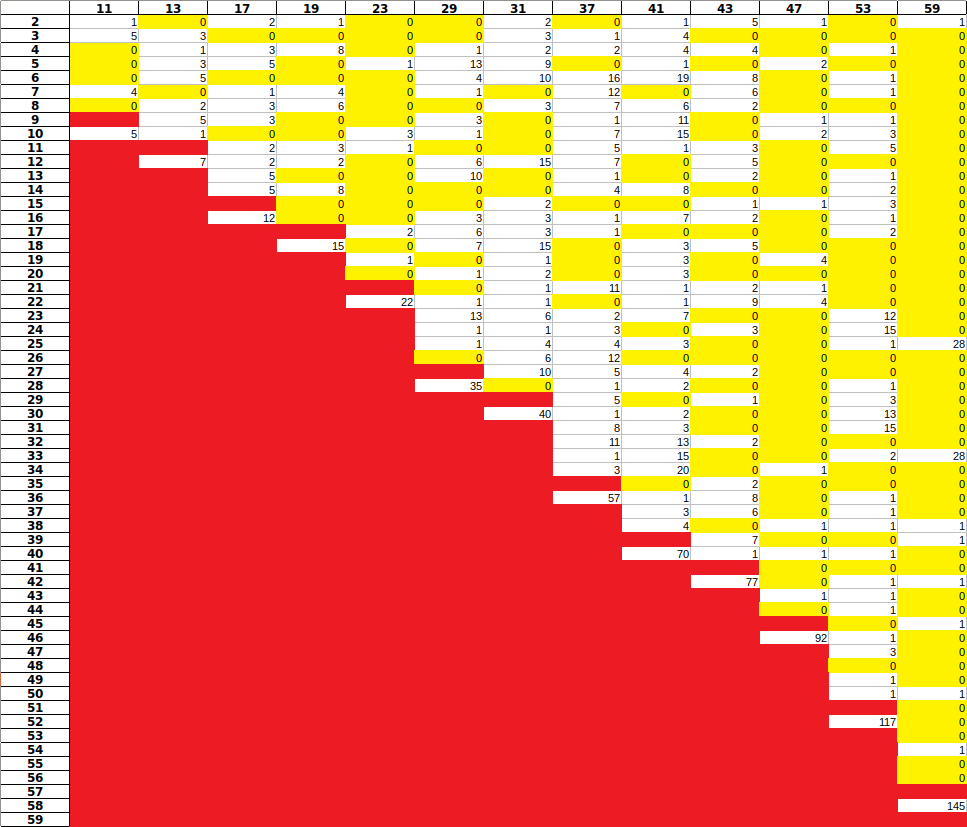}};
    \node[above=3.8, node distance=0cm, yshift=1cm,font=\color{black}] {n (Primes)};
    \node[left=of img1, node distance=0cm, rotate=90, anchor=center,yshift=-0.7cm,font=\color{black}] {c};
    \end{tikzpicture}
    \caption{Data on $X=X_{2,n,c}$ for all primes $11\leq n\leq 59$. A yellow entry means that $X$ is a hypertree. White entries show the (positive) co-dimension of the column space of $A=A_{n,c}$. Red indicates an illegal $c$, i.e., $c\equiv-2$ or $c\geq n$.}
    \label{fig:chunk_of_the_table}
    \centering
\end{figure}

\begin{question}\label{q:main}
For which primes $n$ and $c\in\mathbb{F}_n$ is $X=X_{2,\n,c}$ a hypertree?
\end{question}

Figure \ref{fig:chunk_of_the_table} shows the answer for Question \ref{q:main} for all primes $11\leq n\leq 59$ and each appropriate $c$. Figure \ref{fig:acyclic_percentage} shows the fraction of eligible $c$ for which $X=X_{2,\n,c}$ is a hypertree for all primes $11\leq n\leq1373$.  

\vspace{2mm}
The paper is structured as follows: in Section \ref{section:preliminaries} we discuss some preliminary facts and outline the necessary background in number theory and simplicial combinatorics. It turns out that the problem whether the complex $X$ is acyclic reduces to the question whether a certain matrix with a special structure is invertible. We study this special structure in Section \ref{section:MCB}, explain the reduction and give a new fast algorithm that determines if a matrix of this kind is invertible. In Section \ref{polynomials in the blocks} we further investigate this reduction. This allows us to exhibit in Section \ref{section:Non-acyclic Complexes} an infinite family of non-acyclic $2$-dimensional complexes. We conjecture that a certain simple criterion asymptotically determines if a complex is acyclic or not. Section \ref{section:full_matrices} is devoted to another approach in search of an infinite family of $2$-dimensional hyperpaths.

\section{Preliminaries}\label{section:preliminaries}
Many matrices are defined throughout this paper. They are marked throughout by hyperlinks that can return the reader to their definitions. In Appendix \ref{appendix:_matrix_map} we collect the basic properties of these matrices and their mutual relations. 

\subsection{Some relevant number theory}
For a prime $n$, we denote by $\mathbb{F}_n=\{0,1,\dots,n-1\}$ the field with $n$ elements. Addition and multiplication are done mod $n$. The multiplicative group of $\mathbb{F}_n$ is comprised of the set $\mathbb{F}_n^*=\{1,\dots,n-1\}$. It is a cyclic group isomorphic to $\mathbb{Z}/(n-1)\mathbb{Z}$. The order, $o(x)$ of $x\in\mathbb{F}_n^*$ is the smallest positive integer $r$ s.t.\ $x^r=1$. The following easy lemma gives the orders of $x$'s powers: 

\begin{lemma}\label{lemma_gcd_in_background}
If $n$ is prime and $x\in\mathbb{F}_n^*$, then for every integer $j$
$$o(x^j)=\frac{o(x)}{\gcd\left(j,o(x)\right)}$$
\end{lemma}
\begin{proof}

By definition, $o(x^j)$ is the smallest positive integer $l$ s.t. $x^{j\cdot l}=1$. This exponent $j\cdot l$ must be divisible by $o(x)$, and is therefore the least common multiple of $j$ and $o(x)$. Consequently
$$j\cdot l=\text{lcm}\left(j,o(x)\right)=\frac{j\cdot o(x)}{\gcd\left(j,o(x)\right)}$$
as claimed. 
\end{proof}

Recall Euler's totient function $\varphi$. Namely, $\varphi(t)$ is the number of integers in $\{1,\dots,t-1\}$ that are co-prime with $t$.
It is also the order of the multiplicative group mod $t$.

Clearly $x$ is a generator of $\mathbb{F}_n^*$ iff $o(x)=|\mathbb{F}_n^*|=n-1$. By the above comments, $\mathbb{F}_n^*$ has exactly $\varphi\left(n-1\right)$ generators. We write logarithms w.r.t.\ some fixed generator $\lambda$ of $\mathbb{F}_n^*$. I.e., $\log_\lambda (u)=\log (u)$ is the unique $k\in\{0,\dots,\n-2\}$ for which $\lambda^k=u$.

\subsection{Background on simplicial combinatorics}\label{background_simplicial_combinatorics}

We follows the setup in Chapter 2 of \cite{linial2019extremal}. All simplicial complexes considered here have vertex set $V = \{0, . . . , n-1\}=\mathbb{F}_n$. 
A simplicial complex X is a collection of subsets of V that is closed under taking subsets. Namely, if $\sigma\in X$ and $\tau\subseteq\sigma$, then $\tau\in X$ as well.
Members of $X$ are called faces or simplices. The dimension of the simplex $\sigma\in X$ is defined as $|\sigma|-1$. A $d$-dimensional simplex is also called a $d$-simplex or a $d$-face for short. The dimension $\dim(X)$ is defined as $\max \dim(\sigma)$ over
all faces $\sigma\in X$. The size $|X|$ of a $d$-complex $X$ is the number of its $d$-faces.

The collection of the faces of dimension $\leq\ t$ of $X$, where $t<d$, is called the $t$-skeleton of $X$. We say that a $d$-complex $X$ has a full skeleton if its
$(d-1)$-skeleton contains all the faces of dimensions at most $(d-1)$ spanned by its vertex set. The permutations on the vertices of a face $\sigma$ are split in two orientations of $\sigma$, according to the permutation’s sign. The boundary operator $\partial_d$ maps an oriented $d$-simplex $\sigma=(v_0,\dots, v_d)$ to the formal sum $$\partial_d(\sigma)=\sum_{i=0}^d(-1)^i(\sigma\setminus{v_i})$$ where $\sigma\setminus{v_i}=(v_0,\dots,v_{i-1},v_{i+1},\dots,v_d)$ is an oriented $(d-1)$-simplex. 

We linearly extend the boundary operator to free $\mathbb{Q}$-sums of simplices. We consider the $\binom{n}{d}\times\binom{n}{d+1}$ matrix form of $\partial_d$ by choosing arbitrary orientations for $(d-1)$-simplices and
$d$-simplices. Note that changing the orientation of a $d$-simplex (resp. $d-1$-simplex) results in multiplying the corresponding column (resp.\ row) by $-1$. Thus the $d$-boundary of a weighted sum of $d$-simplices, viewed as a vector $z$ (of weights) of dimension $\binom{n}{d+1}$, is just the matrix-vector product $\partial_dz$.

A simple observation shows that the matrix $\partial_d$ has rank $\binom{n-1}{d}$. We denote by $A$ the submatrix of $\partial_d$ restricted to the columns associated with $d$-faces of a $d$-complex $X$. We define $\text{rank}(X)$ to be $\text{rank}(A)$. The rational $d$-homology of $X$, denoted by $H_d(X; \mathbb{Q})$, is the right kernel of the matrix $A$. Elements of $H_d (X; \mathbb{Q})$ are called $d$-cycles. A $d$-hypertree $X$ over $\mathbb{Q}$ is a $d$-complex of size $\binom{n-1}{d}$ with a trivial rational $d$-dimensional homology. This means that the columns of the matrix $A$ form a basis for the column space of $\partial_d$.

\section{Matrices with Circulant Blocks (MCB)}\label{section:MCB}
It turns out that $A_{\n,c}$ is closely related to a block matrix whose blocks are circulant matrices. So, we start our work on Question \ref{q:main} by deriving a structure theorem for such matrices. This is what we do in the present section.

Recall that a circulant matrix $C\in M_r(\mathbb{Q})$ has the following form

\begin{equation*}
C=
    \begin{pmatrix}
    c_0 & c_{r-1} &  & c_1 \\ 
c_1 & c_0 &  & c_2 \\
 & \ddots & \ddots &  \\ 
 c_{r-1} &  & c_1 & c_0
    \end{pmatrix}
\end{equation*}

Equivalently,
\begin{equation}\label{eq:circulant_as_polynomial}
    C=g(P)=c_0\cdot I+c_1\cdot P+c_2\cdot P^2+\dots+c_{r-1}\cdot P^{r-1}
\end{equation}

where $P\in M_r(\mathbb{Q})$ is the cyclic permutation matrix

\begin{equation}\label{eq:permutation_matrix}
    P=\begin{pmatrix} 0 & 0 &  & 0 & 1 \\
    1 & 0 &  &  & 0 \\
    0 & 1 &  &  &  \\
     &  & \ddots & 0 &  \\
    0 & 0 &  & 1 & 0
\end{pmatrix}  
\end{equation}

Given positive integers $r, t$, we denote by $MCB_{r,t}(\mathbb{Q)}$ (for Matrices with Circulant Blocks) the set of all matrices of the form

\begin{equation}\label{MCB_definition}
E=
    \begin{pmatrix}
    C_{0,0} & C_{0,1} &  & C_{0,t-1} \\ 
C_{1,0} & C_{1,1} &  &  \\
 &  & \vdots &  \\ 
 C_{t-1,0} &  &  & C_{t-1,t-1}.
    \end{pmatrix}
\end{equation}

where each $C_{i,j}\in M_r(\mathbb{Q})$ is a circulant matrix. When $\mathbb{Q}$ is omitted, the matrices are over $\mathbb{C}$.

This is not to be confused with the well-studied class of Circulant Block Matrices (CBM) \cite{davis2013circulant}. Such a matrix is circulant as a block matrix, but its blocks need not be circulants.

We can clearly express $E$ as follows 

\begin{equation}\label{E(P)}
E(P)=\begin{pmatrix}g_{0,0}(P) & g_{0,1}(P) &  & g_{0,t-1}(P) \\ 
g_{1,0}(P) & g_{1,1}(P) &  &  \\
 &  & \vdots &  \\ 
 g_{t-1,0}(P) &  &  & g_{t-1, t-1}(P)
\end{pmatrix} 
\end{equation}

where $g_{i,j}$ are polynomials of degree less than $r$ as in Equation (\ref{eq:circulant_as_polynomial}). Since $P^r=1$, we can view $g_{i,j}(P)$ as elements of the quotient polynomial ring 
\begin{equation}\label{eq:R_polynomial_ring}
    \mathcal{R}:=\mathbb{Q}[P]/(P^r-1)
\end{equation}

Likewise, we think of $E(P)$ as a member in the matrix ring $M_t(\mathcal{R})$. 

Associated with every $z\in\mathbb{C}$ is a {\em scalar} $t\times t$ complex matrix $\underline{E}(z)$, viz.,
\begin{equation}\label{underscore(E)(z)}
\underline{E}(z)=\begin{pmatrix}g_{0,0}(z) & g_{0,1}(z) &  & g_{0,t-1}(z) \\ 
g_{1,0}(z) & g_{1,1}(z) &  &  \\
 &  & \vdots &  \\ 
 g_{t-1,0}(z) &  &  & g_{t-1, t-1}(z),
\end{pmatrix}    
\end{equation}

\begin{theorem}\label{reduction_small_matrices}
A matrix $E\in MCB_{r,t}(\mathbb{Q})$ is singular $\iff$ $\underline{E}(\omega_k)$ is singular for some $k\vert{r}$. Here $\omega_k=exp(\frac{2\pi i}{k})$ is the primitive $k$-th root of unity.
\end{theorem}

\begin{proof} 
The proof of Theorem \ref{reduction_small_matrices} uses the next claim:
\begin{claim}\label{into_block_diagonal}
Every $E\in MCB_{r,t}$ is similar to a block diagonal matrix with $t\times t$ blocks, i.e.,
\begin{equation*}
    X\cdot E\cdot X^{-1} = 
    \begin{pmatrix}
    \underline{E}(\omega_r^r) & 0 & & 0 \\
    0 & \underline{E}(\omega_r^{r-1}) & & 0 \\
    & & \ddots & \\
    0 & 0 & & \underline{E}(\omega_r^1)
    \end{pmatrix}
\end{equation*}
for some invertible matrix $X$.
\end{claim}

\begin{proof}
We recall the order-$r$ Discrete Fourier Transfrom (DFT) Matrix $\mathcal{F}_r$ whose entries are $\mathcal{F}_r[k,l] = \exp(\frac{-2\pi i k l}{r})=\omega_r^{-kl}$, where $\omega_r=\exp(\frac{2\pi i}{r})$ is a primitive $r$-th root of unity. It diagonalizes $r\times r$ circulant matrices as follows:

\begin{lemma}\label{unitary_diagonalizable}
Let $\mathcal{F}_r$ be the order-$r$ DFT matrix. If $C$ is an $r\times r$ circulant matrix, then
\begin{equation*}
    \mathcal{F}_r\cdot C\cdot \mathcal{F}_r^{-1} = \Lambda
\end{equation*}
where $\Lambda$ is the diagonal matrix whose entries are $C$'s eigenvalues, 
\begin{equation*}
    \lambda_j=c_0+c_{r-1}\omega_r^j+c_{r-2}\omega_r^{2j}+\dots+c_1\omega_r^{(r-1)j}~~~~~\text{for}~~~~~0\leq j\leq r-1.
\end{equation*}
Here $(c_0,c_1,\dots,c_{r-1})$ is $C$'s first column.
\end{lemma}

\begin{proof}
$v_j=(1,\omega_r^j,\omega_r^{2j},\dots,\omega_r^{(r-1)j})$ is an eigenvector of $C$ with corresponding eigenvalue $\lambda_j$ as
\begin{gather*}
    (C\cdot v_j)[i]= \\
    =(c_i+c_{i-1}\omega_r^j+c_{i-2}\omega_r^{2j}+\dots+c_0\omega_r^{ij}+c_{r-1}\omega_r^{(i+1)j}+\dots+c_{i+1}\omega_r^{(r-1)j})= \\
    =(c_0+c_{r-1}\omega_r^j+c_{r-2}\omega_r^{2j}+\dots+c_1\omega_r^{(r-1)j})\cdot\omega_r^{ij}=\lambda_j\omega_r^{ij}
\end{gather*}

The claim follows since $\mathcal{F}_r^{-1}=\frac 1r \mathcal{F}_r^{\ast}$
\end{proof}

Set $\mathcal{L}$ as a block diagonal matrix with $\mathcal{F}_r$ on the diagonal. Since $E\in MCB_{r,t}$, Lemma \ref{unitary_diagonalizable} yields  
\begin{equation*}
    \mathcal{L}\cdot E\cdot \mathcal{L}^{-1}=
    \begin{pmatrix}
    \Lambda_{0,0} & \Lambda_{0,1} &  & \Lambda_{0,t-1} \\ 
\Lambda_{1,0} & \Lambda_{1,1} &  &  \\
 &  & \vdots &  \\ 
 \Lambda_{t-1,0} &  &  & \Lambda_{t-1,t-1}
\end{pmatrix} 
\end{equation*}
where $\Lambda_{k,l}=diag(\mathcal{F}_r\cdot c_{k,l})$ is a diagonal $r\times r$ matrix. Here $c_{k,l}$ is the first column vector of the circulant matrix $C_{k,l}$. 

\vspace{2mm}

This matrix has $rt$ rows and columns which we enumerate from $0$ to $rt-1$. It is made up of $r\times r$ blocks, with $t$ of them at every layer. This suggests that indices in this matrix be written as $\alpha r+\beta$ for some $t>\alpha\ge 0$ and $r>\beta\ge 0$, which we interpret as index $\beta$ within block number $\alpha$. We rearrange the matrix to be made up of $t\times t$ blocks, with $r$ of them at every layer, with indices in the form $\gamma t + \delta$ where $r>\gamma\ge 0$ and $t>\delta\ge 0$. Thus the mapping
\begin{equation}\label{permutation_definition}
\varphi:\alpha r+\beta\to \beta t + \alpha
\end{equation}
is a permutation which we apply to the rows and columns of $\mathcal{L}\cdot E\cdot \mathcal{L}^{-1}$.
Since all blocks in $\mathcal{L}\cdot E\cdot \mathcal{L}^{-1}$ are diagonal, the entry in position $(\alpha r+\beta, \alpha' r+\beta')$ is nonzero only if $\beta= \beta'$. Following the application of $\varphi$, the matrix becomes an $r\times r$ diagonal matrix of $t\times t$ blocks.
\begin{equation*}
  Q\cdot\mathcal{L}\cdot E\cdot\mathcal{L}^{-1}\cdot Q^{-1}=
  \begin{pmatrix}
    \Delta_0 & 0 & & 0 \\
    0 & \Delta_1 & & 0 \\
    & & \ddots & \\
    0 & 0 & & \Delta_{r-1}
    \end{pmatrix}    
\end{equation*}
where $Q$ is the permutation matrix of $\varphi$. The matrix thus becomes an $r\times r$ block matrix, with block size $t\times t$, and $$\Delta_{i}[k,l]=\Lambda_{k,l}[i,i]$$ since the mapping (\ref{permutation_definition}) sends

\begin{gather*}
    row\ number\ k\cdot r+i \to row\ number\ i\cdot r+k \\
    column\ number\ l\cdot r+i \to column\ number\ i\cdot r+l \\
\end{gather*}

To complete the proof of Claim \ref{into_block_diagonal}, setting $X=Q\cdot\mathcal{L}$, it only remains to show that $\Delta_i=\underline{E}(\omega_r^{-i})$:

\begin{gather*}
    \Delta_i[k,l]=\Lambda_{k,l}[i,i]=(\mathcal{F}_r\cdot c_{k,l})[i]= \\
    =\sum_{j=0}^{r-1}\mathcal{F}_r[i,j]c_{k,l}[j]=\sum_{j=0}^{r-1}c_{k,l}[j]\cdot\omega_r^{-ij}=\sum_{j=0}^{r-1}c_{k,l}[j]\cdot(\omega_r^{-i})^j= \\ 
    =g_{k,l}(\omega_r^{-i})
\end{gather*}

\end{proof}

For $k\vert{r}$, $\omega_k=\omega_r^{\frac{r}{k}}$. Claim \ref{into_block_diagonal} yields one part of Theorem \ref{reduction_small_matrices}. Namely, that if $\underline{E}(\omega_k)$ is singular for some $E\in MCB_{r,t}$ and some $k\vert{r}$, then $E$ is singular. \\
In order to prove the other direction of Theorem \ref{reduction_small_matrices}, we need the following two Lemmas. Recall from Equation (\ref{eq:R_polynomial_ring}) that for a matrix $E$ in $MCB_{r,t}$, $E(P)$ is a polynomial matrix in $M_t(\mathcal{R})$ over the quotient polynomial ring $\mathcal{R}$. Its determinant $\det(E(P))$ is a polynomial in $\mathcal{R}$, and we denote by $\det(E(P))(z)$ the evaluation of this polynomial at the complex number $z\in\mathbb{C}$. 

\begin{lemma}\label{lemm:determinant_equality}
$\det(E(P))(\omega_r^j)=\det(\underline{E}(\omega_r^j))$ for every $E\in MCB_{r,t}$. 
\end{lemma}

\begin{proof}
\begin{gather*}
    \det(E(P))(\omega_r^j)=\left(\sum_{\sigma\in S_t}\left(\prod_{i=0}^t g_{i,\sigma(i)}(P)\right)\right)(\omega_r^j)= \\
    = \sum_{\sigma\in S_t}\left(\prod_{i=0}^t g_{i,\sigma(i)}(P)\right)(\omega_r^j)
    \numeq{a}
    \sum_{\sigma\in S_t}\prod_{i=0}^t g_{i,\sigma(i)}(P)(\omega_r^j)=\\
    = \sum_{\sigma\in S_t}\prod_{i=0}^t g_{i,\sigma(i)}(\omega_r^j)=
    \det(\underline{E}(\omega_r^j))
\end{gather*}
Equality ($a$) holds, because
$$P^r=(\omega_r^j)^r=1$$
\end{proof}

The next lemma appears without proof in \cite{rjasanow1994effective}. We provide a proof, since we could not find it in the literature: 

\begin{lemma}\label{MCB_group}
The non-singular matrices in $MCB_{r,t}$ form a group w.r.t.\ matrix multiplication.
\end{lemma}

\begin{proof}
Clearly $MCB_{r,t}$ is closed under product, since the product of two circulant matrices is circulant, and matrix multiplication respects block product. We only need to show closure under inverse for invertible matrices in $MCB_{r,t}$.

As mentioned above, $M_t(\mathcal{R})$ and $MCB_{r,t}$ are in one-to-one correspondence. An inverse of $E(P)$ as in Equation (\ref{E(P)}) in $M_t(\mathcal{R})$ is an inverse under this bijection of $E$ in $MCB_{r,t}$, so it remains to prove that if $E$ is invertible, then $E(P)$ has an inverse in $M_t(\mathcal{R})$. 

The determinant of a matrix over a commutative ring is defined as usual as the alternating sum of products over permutations. Such a matrix has an inverse iff its determinant is invertible, as an element of the underlying ring. The proof of this fact (see e.g., \cite{mcdonald1984linear}) goes by establishing the Cauchy–Binet formula for matrices over commutative rings.

We apply this to the commutative polynomial ring $\mathcal{R}$, and conclude that $E(P)$ has an inverse in $M_t(\mathcal{R})$ iff its determinant (which is also a polynomial in $\mathcal{R}$) is invertible. 

To prove the Lemma, let $E\in MCB_{r,t}$ be invertible. Recall the definitions of $E(P)$ in (\ref{E(P)}) and $\underline{E}(\omega_r^j)$ in (\ref{underscore(E)(z)}). By Claim \ref{into_block_diagonal} and since $E$ is invertible, we obtain:

$$\forall j\det(\underline{E}(\omega_r^j))\neq0$$
using Lemma \ref{lemm:determinant_equality} this translates into
$$\forall j\det(E(P))(\omega_r^j)\neq0$$ 
thus $\det(E(P))$ and $P^r-1$ do not share any root and
\begin{equation*}
    \text{gcd}(\det(E(P)),P^r-1)=1
\end{equation*} 
so the determinant is invertible, finishing the proof of the lemma.
\end{proof}
With Lemmas \ref{lemm:determinant_equality} and \ref{MCB_group} we can complete the proof of Theorem \ref{singularity_property_T}. Let $E\in MCB_{r,t}(\mathbb{Q})$ be singular. By Lemma \ref{MCB_group} $E$ has no inverse in $MCB_{r,t}(\mathbb{Q})$, implying that $E(P)$ has no inverse in $M_t(\mathcal{R})$. Consequently, $\det(E(P))$ is not invertible and thus $det(E(P))$ and $P^r-1$ have a non-trivial common divisor. But

\begin{equation*}
    P^r-1=\prod_{k\vert{r}}\Psi_k(P)
\end{equation*}

where

\begin{equation*}
    \Psi_k(P)=\prod_{\substack{1\leq l\leq k \\ 
    gcd(l,k)=1}}(P-\omega_k^l)
\end{equation*}

is the $k$-th cyclotomic polynomial. It is a well known fact that $\Psi_k$ is irreducuble over $\mathbb{Q}$ (e.g., \cite{gauss2006untersuchungen}).
Therefore, $\det(E(P))$ and $P^r-1$ have a non-trivial common divisor iff one of the cyclotomic polynomials divides the determinant. i.e., there exists a divisor $k\vert{r}$ s.t.\
$$\Psi_k(P)\big\vert{\det(E(P))}$$ Since $\Psi_k(\omega_k)=0$ 
there exists a divisor $k$ of ${r}$, s.t.\ $$\det(E(P))(\omega_k)=0$$ By Lemma \ref{lemm:determinant_equality} this implies
\begin{equation*}
    \det(\underline{E}(\omega_k))=0
\end{equation*}
completing the proof of Theorem \ref{reduction_small_matrices}.
\end{proof}

\subsection{Computational Aspects of MCB}
Theorem \ref{reduction_small_matrices} has interesting computational aspects. In order to present them, we need some preparations. We recall that
 $d(m)$ denotes the number of distinct divisors of the integer $m$, and that $d(m)=m^{o_m(1)}$, more precisely (see \cite{hardy1979introduction})
 $$\limsup_{n\to\infty}\frac{\log d(n)}{\log n/\log\log n}=\log2$$

\begin{note}\label{not_equivalent_problems}
It is a classical fact (e.g., \cite{petkovic2009generalized}) that matrix multiplication and matrix inversion have essentially the same computational complexity. It is, however, still unknown if these problems are also equivalent to the decision problem whether a given matrix is invertible (e.g., \cite{Blaser_matrix_survey}).
\end{note}

The smallest exponent for matrix multiplication is commonly denoted by $\omega$. This is the least real number such that two $n\times n$ matrices can be multiplied using $O(n^{\omega+\epsilon})$ arithmetic operations for every $\epsilon>0$. Presently the best known bounds \cite{le2014powers} are $2\leq\omega \leq2.373$. 

\begin{proposition}\label{corollary_T}
For every $\epsilon>0$, it is possible to determine in time $$O\left(r^{1+\epsilon}\cdot t^2+r^{\epsilon}\cdot t^{\omega}\right)$$ whether 
a matrix in $MCB_{r,t}(\mathbb{Q})$ is invertible.
\end{proposition}

\begin{note}\label{note:time_comparison}
It follows that when $r\to\infty$, it is easier to decide the invertibility of matrices in $MCB_{r,t}(\mathbb{Q})$ than general $rt\times rt$ matrices, because \[(rt)^{\omega}\gg r^{1+\epsilon}\cdot t^2+r^{\epsilon}\cdot t^{\omega}\]
There is an obvious lower bound of $\Omega(r\cdot t^2)$, which is
the time it takes to read a matrix in $MCB_{r,t}(\mathbb{Q})$.
\end{note}

\begin{proof}(Proposition \ref{corollary_T})
The proof of theorem \ref{reduction_small_matrices} yields an algorithm to decide if $E\in MCB_{r,t}(\mathbb{Q})$ is invertible:
\begin{enumerate}
    \item Produce the matrix $E(P)$ as in equation (\ref{E(P)}).
    \item For each divisor $k$ of $r$:
    \begin{enumerate}[label*=\arabic*.]\label{evaluating small matrices}
        \item Calculate the matrix $\underline{E}(\omega_k)$ as in equation (\ref{underscore(E)(z)}) by evaluating the polynomial matrix $E(P)$ with $\omega_k=\exp(\frac{2\pi i}{k})$. \label{making_small_matrix}
        \item Determine if the $t\times t$ matrix $\underline{E}(\omega_k)$ is invertible. If it is singular, return '$E$ is singular'. \label{using_T}
    \end{enumerate}
    \item If $\underline{E}(\omega_k)$ has full rank for every divisor $k\vert r$, return '$E$ is invertible'.
\end{enumerate}
A circulant block is clearly completely defined by its first row. Therefore the matrix $E(P)$ can be found in time $O(r\cdot t^2)$. To find all the divisors of $r$ we can even factor $r$ using Eratosthenes' sieve with no harm to the complexity. Step \ref{evaluating small matrices} is repeated $d(r)=r^{o(1)}$ times. Horner's Rule allows to evaluate a degree $r$ polynomial with $r$ additions and $r$ multiplications, so step \ref{making_small_matrix}\ takes time $O(r\cdot t^2)$. The running time of step \ref{using_T} is at most $O(t^{\omega})$. All told the combined running time is $$O\left(r\cdot t^2+r+d(r)\cdot(r\cdot t^2+t^{\omega})\right)=O\left(r^{1+\epsilon}\cdot t^2+r^{\epsilon}\cdot t^{\omega}\right)$$
for every $\epsilon>0$. 
\end{proof}

In step \ref{evaluating small matrices} we need to check whether $\underline{E}(\omega_k)$ is invertible for different $k$. These calculations can clearly be done in parallel.

In \cite{tsitsas2007recursive} an iterative algorithm is presented to invert a matrix in $MCB_{r,t}$, with run time

\begin{equation}\label{running_time_tsitsas}
    \begin{cases}
    O(2^{3l}\cdot r + t^2\cdot r^2) & r \text{~is not a power of~} 2\\
    O(2^{3l}\cdot r + t^2\cdot r\log r) & r \text{~is a power of~} 2
    \end{cases}
\end{equation}
where $l:=\lceil\log_2t\rceil$. If we only need to decide whether the matrix is invertible, then the algorithm in Corollary \ref{corollary_T} is faster. We still do not know whether Theorem \ref{reduction_small_matrices} yields an algorithm to invert a matrix in $MCB_{r,t}(\mathbb{Q})$ that is faster then the algorithm from \cite{tsitsas2007recursive}.

\subsection{From $A$ to $MCB$.}\label{from_A_to_MCB_S}
It turns out that there is a rank-preserving transformation of the boundary operator $A=A_{n,c}$ as defined in section \ref{background_simplicial_combinatorics} into a matrix in $MCB_{n-1,\frac{n-3}{2}}$ as defined in (\ref{MCB_definition}). The transformation is fairly simple and only involves reordering of the rows and columns plus removal of $n-1$ rows that are linearly dependent on the other rows and a trivial Gauss elimination of $\frac{n-1}{2}$ rows and columns. 

As mentioned in section \ref{background_simplicial_combinatorics}, and maintaining the same terminology, the boundary operator $A=A_{n,c}$ of $X=X_{2,\n,c}$ is given by an $\binom{n}{2} \times\binom{n-1}{2}$ matrix. We next find a square $\binom{n-1}{2}\times \binom{n-1}{2}$ submatrix of $A$ of rank $\text{rank}(A)$. To this end, we remove $(\n-1)$ rows of $A$ which are linearly spanned by the other $\binom{n-1}{2}$ rows. Rows in $A$ are indexed by edges ($1$-dimensional faces). It is well known and easy to prove that this is the case with any $n-1$ rows that represent the edge set of a spanning tree of the complete graph $K_{\n}$. We apply this with the star rooted at vertex $0$. In other words, we remove the rows corresponding to pairs $\{\{0,j\}\suchthat j\in \mathbb{F}_{\n}^{\ast}\}$.

Nonzero elements of $\mathbb{F}_n$ act on subsets of $\mathbb{F}_n$. Namely, if $u\in\mathbb{F}_n^*$, and $\sigma=\{x_0,\dots,x_k\}\subseteq \mathbb{F}_n$ we denote:
\begin{equation*}
    u\cdot \sigma=u\cdot \{x_0,\dots,x_k\}:=\{u\cdot x_0,\dots,u\cdot x_k\}.
\end{equation*}
We note that $X$ is closed under such action, and linearly extend this definition as $u\cdot(\sigma_1+\sigma_2)=u\cdot\sigma_1+u\cdot\sigma_2$. 

We organize $A$'s rows and columns by blocks corresponding to the orbits under the action of $\mathbb{F}_n^*$. A simple calculation shows that $X$'s $2$-faces form $\frac{n-3}{2}$ orbits of size $n-1$, and one orbit of size $\frac{n-1}{2}$. The latter is comprised of all $2$-faces of the form $u\cdot\{1,-1,0\}$. The $1$-faces (edges) also form $\frac{n-3}{2}$ orbits of size $n-1$, and one orbit of size $\frac{n-1}{2}$ that includes the edges $u\cdot\{1,-1\}$.

For $x\in\mathbb{F}_n^{\ast}$, and $y\in\mathbb{F}_n$, the block called $B_{[x,y]}$ is characterized by having the edge $\{1,x\}$ as a row and the $2$-face $\{1,y,-\frac{1+y}{c}\}$ as a column. We refer to $x$ and $y$ as the row and column {\em leaders} of $B$. This creates some ambiguity, since $\{1,x\}$ and $\{1, x^{-1}\}$ belong to the same block. Between $x$ and $x^{-1}$ the leader is the one with the smaller logarithm (as defined in Section \ref{section:preliminaries}). The same ambiguity and the way around it apply as well to $y$ and $y^{-1}$. We order the rows of the block indexed by $x$ as follows:
 
 \begin{equation}\label{eq:orientation_edges}
     [(1,x),\lambda\cdot(1,x),\lambda^2\cdot(1,x),\dots,\lambda^{n-2}\cdot(1,x)].
 \end{equation}
Likewise, the columns in a block whose column index is $y$ are ordered as follows:
\begin{equation}\label{eq:orientation_faces}
     [(1,y,z),\lambda\cdot(1,y,z),\lambda^2\cdot(1,y,z),\dots,\lambda^{n-2}\cdot(1,y,z)]
 \end{equation}
where $z=-\frac{1+y}{c}$. Equations (\ref{eq:orientation_edges}) and (\ref{eq:orientation_faces}) also represent the orientation that we use for the edges and $2$-faces, indicated by the use of tuples over sets. 

Note that the column of a $2$-face $\{u,-u,0\}$ has a single non-zero entry in row $\{u,-u\}$, since we have removed the rows that correspond to the star centered at vertex $0$. We eliminate these $\frac{n-1}{2}$ rows and columns by (trivial) Gauss elimination as in \cite{aronshtam2013collapsibility} and arrive at 
\begin{equation}\label{eq:S}
    S=S_{n,c}
\end{equation} 
an $\frac{n-3}{2}\cdot (n-1)\times\frac{n-3}{2}\cdot (n-1)$ submatrix of $A$. To recap, $X$ is a hypertree iff $S$ is non-singular. 

\begin{claim}
Every block $B_{[x,y]}$ of $S$ is circulant, i.e., $S\in MCB_{n-1,\frac{n-3}{2}}(\mathbb{Q})$.
\end{claim}

\begin{proof}

The boundary operator is a signed inclusion matrix where the column that corresponds to the oriented face $(u,v,w)$ is
$$e_{(u,v)}-e_{(u,w)}+e_{(v,w)}$$
Where for an oriented edge $(u,v)$ we define $e_{(u,v)}$ to be the $0/1$ column vector with a single 1 in position $(u,v)$. If the direction is opposite, i.e. the edge $(v,u)$ is present and not $(u,v)$, then $e_{(v,u)}=-e_{(u,v)}$. Note that the column that corresponds to the oriented face $(1,y,z)$ is
\begin{equation}\label{eq:first_boundary}
    S[:,(1,y,z)]=e_{(1,y)}-e_{(1,z)}+e_{(y,z)}
\end{equation}
while the column that corresponds to the oriented face in the same block\\ 
$\lambda^k\cdot(1,y,z)$ is
\begin{equation}\label{eq:second_boundary}
    S[:,\lambda^k\cdot(1,y,z)]=e_{\lambda^k\cdot(1,y)}-e_{\lambda^k\cdot(1,z)}+e_{\lambda^k\cdot(y,z)}
\end{equation}

To show the blocks in $S$ are circulant, We need to show that
\begin{equation*}
    B_{[x,y]}[i,0]=B_{[x,y]}[i+k,k]
\end{equation*}
Where addition is done modulo $n-1$. Before we give this proof, let us illustrate it with our running example:

\subsection*{Running Example}
Our chosen generator of $\mathbb{F}_{13}^*$ is $\lambda=2$. The next matrix $B_{[3,4]}$ is an example of a block matrix in $S$ whose row leader is $3$ and column leader $4$:

\medskip
$$B_{[3,4]}=$$
\begin{adjustwidth}{-1in}{-1in}

\begin{tabular}{|l|ccccccc|}
\hline
\diagbox{\text{Edge}}{\text{2-Face}} & $(1,4,12)$ & $\begin{array}{@{}c@{}}2\cdot(1,4,12) \\ =(2,8,11)\end{array}$ & 
$\begin{array}{@{}c@{}}2^2\cdot(1,4,12) \\ =(4,3,9)\end{array}$
 & 
... &
$\begin{array}{@{}c@{}}2^9\cdot(1,4,12) \\ =(5,7,8)\end{array}$ &
$\begin{array}{@{}c@{}}2^{10}\cdot(1,4,12) \\ =(10,1,3)\end{array}$ &
$\begin{array}{@{}c@{}}2^{11}\cdot(1,4,12) \\ =(7,2,6)\end{array}$ 
\\
\hline
(1,3) & 0 & 0 & 0 & \dots & 0 & 1 & 0 \\
$2\cdot(1,3)=(2,6)$ & $0$ & $0$ & $0$ &  & $0$ & $0$ & $1$ \\
$2^2\cdot(1,3)=(4,12)$ & $1$ & $0$ & $0$ &  & $0$ & $0$ & $0$  \\
$2^3\cdot(1,3)$=$(8,11)$ & $0$ & $1$ & $0$ &  & $0$ & $0$ & $0$  \\
$2^4\cdot(1,3)=(3,9)$ & $0$ & $0$ & $1$ & & $0$ & $0$ & $0$  \\
$\vdots$ & $\vdots$ &  &  & $\ddots$ & $0$ & $0$ & $0$ \\
$2^{11}\cdot(1,3)=(7,8)$ & $0$ & $0$ & $0$ & $\dots$ & $1$ & $0$ & $0$ \\
\hline
\end{tabular}

\end{adjustwidth}

\vspace{2mm}

For example, the leader column of $(1,4,12)$ has non-zero entries in rows $(1,4),(1,12)$ and $(4,12)$. In $B_{[3,4]}$ only row $(4,12)$ is present. Indeed $B_{[3,4]} = P^2$ is a circulant matrix.

\vspace{2mm}
\line(1,0){350}

\begin{gather*}
    B_{[x,y]}[i,0]=S[\lambda^i\cdot(1,x),(1,y,z)]=S[\lambda^{i+k}\cdot(1,x),\lambda^k(1,y,z)]= \\
    =B_{[x,y]}[i+k,k]
\end{gather*}
The second equality stems from the indexation of the rows and $2$-faces in the block as in Equations (\ref{eq:orientation_edges}), (\ref{eq:orientation_faces}) and the definition of the boundary matrix for the appropriate columns in Equations (\ref{eq:first_boundary}), (\ref{eq:second_boundary}). 
 
The fact that the blocks $B_{[x,y]}$ are circulant, follows also from the invariance of the boundary operator under the action of $\mathbb{F}_n^*$. 
\begin{equation*}
    \lambda\cdot\partial_2\left(\sigma\right)=\partial_2\left(\lambda\cdot\sigma\right)
\end{equation*}
for
\begin{gather*}
    \lambda\cdot\partial_2\left(\sigma\right) = \lambda\cdot\left((u_0,u_1)-(u_0,u_2)+(u_1,u_2)\right)=\\
    = (\lambda\cdot u_0,\lambda\cdot u_1)-(\lambda\cdot u_0,\lambda\cdot u_2)+(\lambda\cdot u_1,\lambda\cdot u_2)= \\
    =\partial_2(\lambda\cdot u_0,\lambda\cdot u_1,\lambda\cdot u_2) = \partial_2\left(\lambda\cdot\sigma\right)
\end{gather*}

As $S$ is a submatrix of the matrix form of $\partial_d$, the claim follows from the indexation of the rows and columns.

\end{proof}
\section{$S$ as a polynomial matrix}\label{polynomials in the blocks}
This section we find out how to express the matrix $S=S_{n,c}$ as defined in (\ref{eq:S}) as a polynomial matrix.
This matrix is a sum of three sparse matrices, each having at most one nonzero block per row and column. All these blocks have the form $\pm P^j$ for some $n-2\ge j\ge 0$, because the boundary operator is a signed inclusion matrix. Thus the column of the oriented triple $(1,y,z)$ may only contain terms corresponding to $(1,y)$, to $(1,z)$ and to $(y,z)$. But any particular row may be labeled by either $x$ or $x^{-1}$, so there are (at most) three possible cases where the block $B_{[x,y]}$ is nonzero indexed as follows
\begin{itemize}
\item $i=1$, $y\in \{x, x^{-1}\}$
\item $i=2$, $z\in \{x, x^{-1}\}$
\item $i=3$, $zy^{-1}\in \{x, x^{-1}\}$.
\end{itemize}

In case $i$ we refer to the relevant $x$ as $x_i$. To sum up

\begin{equation*}
B_{[x,y]}=T_1+T_2+T_3
\end{equation*}

where

\begin{equation*}
    T_1=
    \begin{cases}
        I & y=x \\
        -P^{-\log x} & y=x^{-1} 
    \end{cases} 
\end{equation*}

\begin{equation*}
    T_2=
    \begin{cases}
        -I & z=x \\
        P^{-\log x} & z=x^{-1}
    \end{cases}
\end{equation*}

\begin{equation*}
    T_3=
    \begin{cases}
        P^{\log y} & z\cdot y^{-1}=x \\
        -P^{\log z} & z\cdot y^{-1}=x^{-1} 
    \end{cases} 
\end{equation*}
    
Most columns indeed have exactly three nonzero terms, although two of them possibly reside in the same block. We return again to our running example for illustration.

\subsection*{Running Example}
The next table shows the logarithm to base $\lambda=2$ of every $x\in \mathbb{F}_{13}^*$, and its order $o(x)$:

\begin{equation*}
    \begin{tabular}{|c|cccccccccccc|}
    \hline
    x & 1 & 2 & 4 & 8 & 3 & 6 & 12 & 11 & 9 & 5 & 10 & 7 \\
    \hline
    log(x) & 0 & 1 & 2 & 3 & 4 & 5 & 6 & 7 & 8 & 9 & 10 & 11 \\
    \hline
    o(x) & 1 & 12 & 6 & 4 & 3 & 12 & 2 & 12 & 3 & 4 & 6 & 12 \\
    \hline
    \end{tabular}
\end{equation*}
\vspace{1mm}

We next present the matrix $S=S_{13,5}$ in block form. To the left and above the matrix appear the blocks' edge and $2$-face leaders. For example, \mbox{$(1,4,12)\in X$}, since $1+4+5\cdot12\equiv0\ mod\ 13$ and it is the $2$-face leader of the first column. No leader has the form $(1,2,z)$, since $1+2+5\cdot2\equiv0$, i.e., $z=y=2$.

\begin{equation*}
    S=
    \begin{tabular}{|l|ccccc|}
\hline
\diagbox{Edge}{2-Face} & (1,4,12) & (1,8,6) & (1,3,7) & (1,6,9) & (1,0,5) \\
\hline
(1,2) & $0$ & $0$ & $P^{11}$ & $0$ & $0$ \\
(1,4) & $I$ & $P^{3}$ & $0$ & $0$ & $0$ \\
(1,8) & $0$ & $I$ & $0$ & $P^{5}$ & $P^{9}$ \\
(1,3) & $P^{2}$ & $0$ & $I$ & $P^{8}$ & $0$ \\
(1,6) & $0$ & $-I$ & $-P^{11}$ & $I$ & $0$ \\
\hline
\end{tabular}
\end{equation*}
\vspace{1mm}

The block $B_{[3,4]}$ from our previous discussion of the running example resides in the first column and fourth row. The matrix
$S$ is $60\times 60$, being a $5\times 5$ matrix of $12\times 12$ blocks, whence $\underline{S}(z)$ is a $5\times5$ matrix. For example where $z=\omega_3$:

\begin{equation*}
    \underline{S}(\omega_3)=
    \begin{tabular}{|l|ccccc|}
\hline
\diagbox{Edge}{2-Face} & (1,4,12) & (1,8,6) & (1,3,7) & (1,6,9) & (1,0,5) \\
\hline
(1,2) & $0$ & $0$ & $\omega_3^{2}$ & $0$ & $0$ \\
(1,4) & $1$ & $1$ & $0$ & $0$ & $0$ \\
(1,8) & $0$ & $1$ & $0$ & $\omega_3^{2}$ & $1$ \\
(1,3) & $\omega_3^{2}$ & $0$ & $1$ & $\omega_3^{2}$ & $0$ \\
(1,6) & $0$ & $-1$ & $-\omega_3^{2}$ & $1$ & $0$ \\
\hline
\end{tabular}
\end{equation*}
\vspace{2mm}
\hrule

\section{Non-acyclic Complexes}\label{section:Non-acyclic Complexes}
In this section we present an infinite family of non-acyclic $2$-dimensional complexes. As subsequently discussed in the section, we suspect that asymptotically almost every non-acyclic complex is in this family and asymptotically almost every complex that is not in this family is acyclic.  

Before we move on to the main subject of this section, we prove the following curious connection between $o(c)$ and $\log(c)$:

\begin{lemma}\label{lemma:equal_gcd}
For $n$ prime and $c\in\mathbb{F}_n^*$, $\text{gcd}\left(\log(c),\frac{n-1}{2}\right)=\text{gcd}\left(\frac{n-1}{o(c)},\frac{n-1}{2}\right)$
\end{lemma}

\begin{proof}
Let $k_1=\text{gcd}\left(\log(c),\frac{n-1}{2}\right)$ and $k_2=\text{gcd}\left(\frac{n-1}{o(c)},\frac{n-1}{2}\right)$
\begin{itemize}
    \item $k_2\big\vert{k_1}$:
    
    $\lambda^{\log(c)}=c$ so $\lambda^{\log(c)\cdot o(c)}=1$. But $\lambda$ is a generator whence $$(n-1)\big\vert\log(c)\cdot o(c)$$ thus 
    
    \begin{equation}\label{eq:k_2_vert_k_1_eq}
        \frac{n-1}{o(c)}\Big\vert\log(c)
    \end{equation}
    
    Since $k_2\vert{\frac{n-1}{o(c)}}$ and $k_2\vert{\frac{n-1}{2}}$, from equation (\ref{eq:k_2_vert_k_1_eq}) we conclude that 
    $$k_2\big\vert{\text{gcd}\left(\log(c),\frac{n-1}{2}\right)}=k_1$$
    \item $k_1\big\vert k_2:$
    
It is sufficient to prove that $k_1\vert\frac{n-1}{o(c)}$. We apply Lemma \ref{lemma_gcd_in_background} with $x=\lambda^2$ and $j=\log(c)$ and conclude that
    \begin{equation}\label{plug_in_gcd}
        o\left(\lambda^{2\log(c)}\right)=\frac{o(\lambda^2)}{\gcd\left(\log(c),o(\lambda^2)\right)}    
    \end{equation}
But $o(\lambda^2)=\frac{n-1}{2}$ and $o\left(\lambda^{2\log(c)}\right)=o(c^2)$. Rearranging Equation (\ref{plug_in_gcd}) gives
    $$k_1=\gcd\left(\log(c),\frac{n-1}{2}\right)=\frac{n-1}{2\cdot o(c^2)}$$
    Lemma \ref{lemma_gcd_in_background} also tells us that if $o(c)$ is even, $o(c^2)=\frac{o(c)}{2}$ and if $o(c)$ is odd, $o(c^2)=o(c)$. In either case $k_1\vert\frac{n-1}{o(c)}$, as required.
\end{itemize}
\end{proof}

\begin{corollary}\label{corrolary_gcd}
For $n$ prime and $c\in\mathbb{F}_n^*$, $\gcd\left(\log(-c), \frac{n-1}{2}\right)=\gcd\left(\frac{n-1}{o(c)},\frac{n-1}{2}\right)$
\end{corollary}
\begin{proof}
It is sufficient to prove that $\gcd\left(\log(-c), \frac{n-1}{2}\right)=\gcd\left(\log(c), \frac{n-1}{2}\right)$.
Since $\log(-1)=\pm\frac{n-1}{2}$ it follows that $\log(-c)=\log(c)\pm\frac{n-1}{2}$. This yields
$$\gcd\left(\log(-c), \frac{n-1}{2}\right)=\gcd\left(\log(c)\pm\frac{n-1}{2}, \frac{n-1}{2}\right)=\gcd\left(\log(c), \frac{n-1}{2}\right)$$
as claimed.
\end{proof}

\begin{theorem}\label{singularity_property_T}
The complex $A_{n,c}$ is non-acyclic in the following cases:
\begin{itemize}
\item
$n\equiv 1\bmod 4$ and $c$ is not a primitive element of $\mathbb{F}_n$
\item
$n\equiv 3\bmod 4$ and $c$ is neither primitive nor of order $\frac{n-1}{2}$
\end{itemize}
\end{theorem}

\begin{proof}

Note that in both cases $\gcd\left(\frac{n-1}{o(c)},\frac{n-1}{2}\right)>1$. When $n\equiv 1\bmod 4$, the integer $\frac{n-1}{2}$ has all the prime factors of $n-1$, so $\frac{n-1}{o(c)}$ and $\frac{n-1}{2}$ are relatively prime only if $o(c)=n-1$, i.e., $c$ is primitive. When $n\equiv 3\bmod 4$, the only divisor of $n-1$ that fails to divide $\frac{n-1}{2}$ is $2$. Therefore $\frac{n-1}{o(c)}$ and $\frac{n-1}{2}$ are relatively prime only if $c$ is primitive or $o(c)=\frac{n-1}{2}$.

So let $k\ge 2$ be a common divisor of $\frac{n-1}{o(c)}$ and $\frac{n-1}{2}$.
We show below that the vector $v_{n,k}$ whose $x$ entry is

\begin{equation*}
    v_{n,k}[x]=1-\omega_k^{\log(x)}
\end{equation*}
is in the left kernel of $\underline{S}(\omega_k)$. We first illustrate this in our running example:

\subsection*{Running Example}
Theorem \ref{singularity_property_T} applies to our running example, since $o(5)=4$, and $\text{gcd}(\frac{12}{4},\frac{12}{2})=3>1$. It is easily verified that the following vector is in the left kernel of $\underline{S}(\omega_3)$:
\begin{equation*}
    v_{13,3}=
    \begin{pmatrix}
    1-\omega_3^1 & 1-\omega_3^2 & 0 & 1-\omega_3^1 & 1-\omega_3^2 
    \end{pmatrix} 
\end{equation*}

\begin{equation*}
    \begin{pmatrix}
    1-\omega_3^1 & 1-\omega_3^2 & 0 & 1-\omega_3^1 & 1-\omega_3^2 
    \end{pmatrix} \cdot
    \begin{pmatrix}
    0 & 0 & \omega_3^{2} & 0 & 0 \\
1 & 1 & 0 & 0 & 0 \\
0 & 1 & 0 & \omega_3^{2} & -1 \\
\omega_3^{2} & 0 & 1 & \omega_3^{2} & 0 \\
0 & -1 & -\omega_3^{2} & 1 & 0 \\
    \end{pmatrix}=\vv{0}
\end{equation*}
\vspace{2mm}
\hrule
\vspace{2mm}

For $y\in \mathbb{F}_n^*$ we compute the $y$-th entry of $v_{n,k}\cdot\underline{S}(\omega_k)$, i.e., the coordinate that corresponds to the $2$-face $(1,y,z)$ 

\begin{equation}\label{dot_product_with_a_row}
    (v_{n,k}\cdot\underline{S}(\omega_k))[y]=\sum_{i=1,2,3}(1-\omega_k^{\log(x_i)})\cdot T_i(\omega_k)
\end{equation}

We define $\Theta_{k,i}$ as the scalar term that is obtained upon evaluating $T_i$ at $\omega_k$.

\begin{lemma}\label{different_cases_are_equal}
Let $t_{k,i}:=(1-\omega_k^{\log(x_i)})\cdot \Theta_{k,i}$ be the $i$-th term in (\ref{dot_product_with_a_row}). Then
\begin{enumerate}
    \item $t_{k,1}=1-\omega_k^{\log(y)}$
    \item $t_{k,2}=\omega_k^{\log(z)}-1$
    \item $t_{k,3}=\omega_k^{\log(y)}-\omega_k^{\log(z)}$
\end{enumerate}
\end{lemma}

\begin{proof} Let us go through the cases:

\begin{enumerate}
    \item If $x_1=y$, then $\Theta_{k,1}=1$\\
    If $x_1=y^{-1}$, then $\Theta_{k,1}=-\omega_k^{-\log x_1}$.
    
    \item If $x_2=z$, then $\Theta_{k,2}=-1$\\
    If $x_2=z^{-1}$, then $\Theta_{k,2}=\omega_k^{-\log x_2}$.
    
    \item If $x_3=z\cdot y^{-1}$, then $\Theta_{k,3}=\omega_k^{\log y}$\\
    If $x_3=z^{-1}\cdot y$, then $\Theta_{k,3}=-\omega_k^{\log z}$.  
\end{enumerate}
All cases are readily verifiable.
\end{proof}

It is easy to check that $t_{k,1}+t_{k,2}+t_{k,3}=0$ for those columns $y$ where $T_1, T_2, T_3$ are well defined. We turn to deal with the exceptional cases where some of the definitions fail.

The matrix $T_2$ is undefined for the column of $y=c-1, z=-1$, since row $(1,-1)$ is absent. But then

$$t_{k,1}+t_{k,3}=1-\omega_k^{\log(y)}+\omega_k^{\log(y)}-\omega_k^{\log(z)}
    =1-\omega_k^{\log(-1)}=1-\omega_k^{\frac{n-1}{2}}=0,$$

since $k\vert{\frac{n-1}{2}}$.

Also, $T_1, T_3$ are undefined when $y=0, z=-c^{-1}$. However, in this case $t_{k,2}=\omega_k^{-\log(-c)}-1=0$, a deduction from Corollary \ref{corrolary_gcd} since $k\big\vert{\log(-c)}$.

\end{proof}

\subsection{A Conjecture about a single criterion}

Again let $n$ be a prime, and $c\neq 0,\pm 1, -2$. As illustrated in Figure \ref{fig:acyclic_probability},
our computer simulations suggest that for large $n$ Theorem \ref{singularity_property_T} captures asymptotically almost all cases in which $X_{2,n,c}$ is non-acyclic. For the acyclic cases, asymptotically almost all other cases that Theorem \ref{singularity_property_T} does not capture are acyclic. 

\begin{figure}[!htb]
    \centering
    \includegraphics[width=1\textwidth]{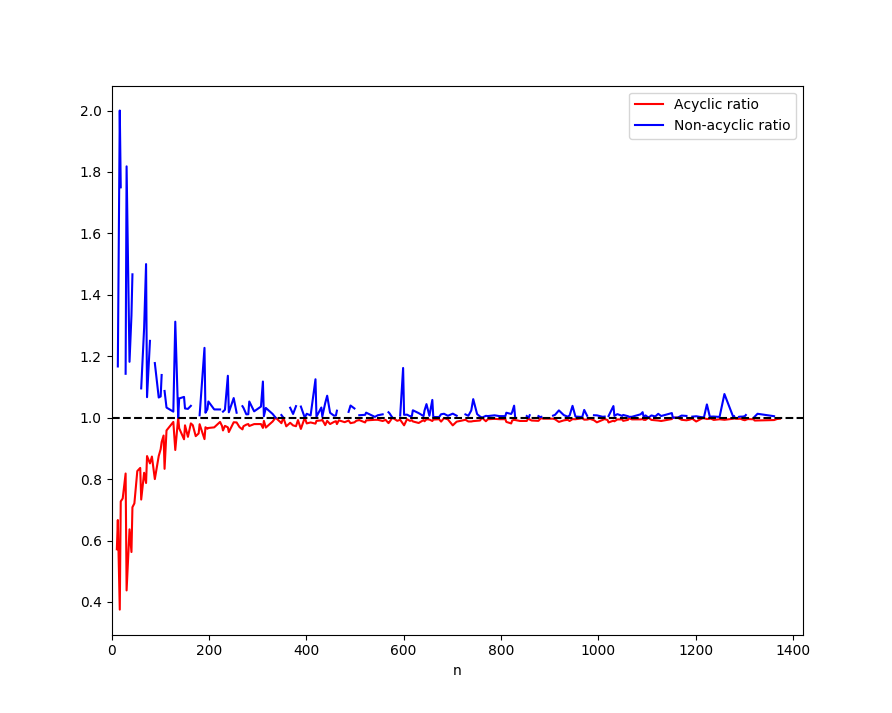}
    \caption{
    Acyclic and Non-Acyclic ratios for all primes $11\leq n\leq 1373$. The acyclic ratio is $A_n/(n-4-E_n)$ and the non-acyclic ratio is $N_n/E_n$. Here $A_n$ and $N_n$ are the number of $c\in\mathbb{F}_n^*\setminus\{\pm1,-2\}$ for which $X_{2,n,c}$ is acyclic, resp.\ non-acyclic. $E_n$ is the number of non-acyclic complexes explained by Theorem \ref{singularity_property_T}. 
    } 
    \label{fig:acyclic_probability}
\end{figure}

This suggests that the following simple criterion asymptotically determines whether or not a complex is acyclic. The asymptotics is w.r.t.\ $n\to\infty$.

\begin{conjecture}\label{conj}
Let the complex $X_{2,n,c}$ with $n$ prime and $c\in\mathbb{F}_n^*$ be as above.
\begin{itemize}
\item
When $n\equiv 1\bmod 4$:
\begin{itemize}
\item 
Recall that $X_{2,n,c}$ is non-acyclic if $c\in\mathbb{F}_n^*$ is non-primitive.
\item
We conjecture that $X_{2,n,c}$ is acyclic for asymptotically almost every primitive $c\in\mathbb{F}_n^*$.
\end{itemize}
\item
When $n\equiv 3\bmod 4$:
\begin{itemize}
\item 
Recall that $X_{2,n,c}$ is non-acyclic if $c\in\mathbb{F}_n^*$ is neither primitive, nor of order $\frac{n-1}{2}$.
\item
We conjecture that $X_{2,n,c}$ is acyclic for asymptotically almost every $c\in\mathbb{F}_n^*$ that is either primitive or of order $\frac{n-1}{2}$.
\end{itemize}
\end{itemize}
\end{conjecture}

In light of Figure \ref{fig:acyclic_percentage}, we can state Conjecture \ref{conj} in terms of acyclic ratios:

\begin{figure}[!htb]
    \centering
    \includegraphics[width=0.95\textwidth]{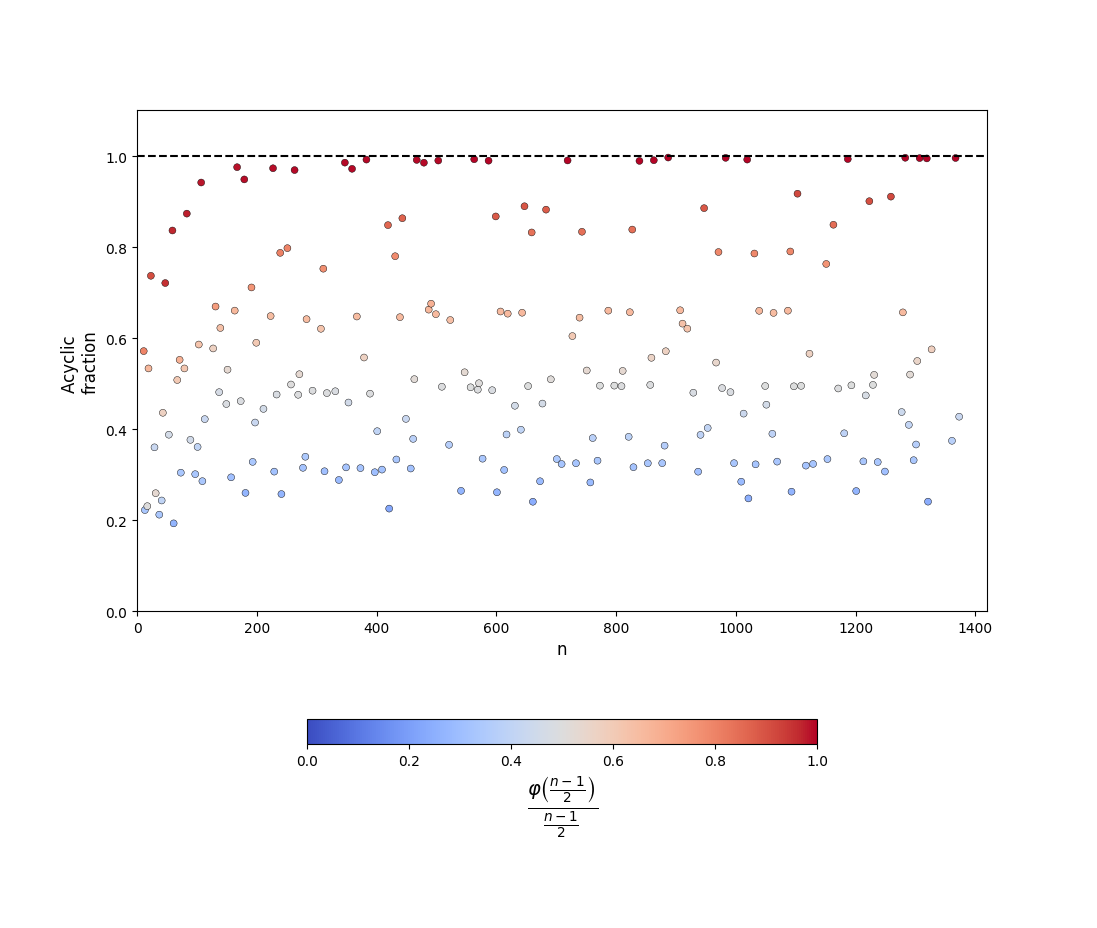}
    \caption{
    For each prime $11\leq n\leq 1373$ this is the probability that $X_{2,n,c}$ is acyclic over all eligible $c\in\mathbb{F}_n^*\setminus\{\pm 1,-2\}$.}
    
    \label{fig:acyclic_percentage}
\end{figure}

\begin{claim}\label{conjecture_claim}
By Theorem \ref{singularity_property_T}, for every prime $n$ the acyclic ratio (i.e., the fraction of complexes $X_{2,n,c}$ that are acyclic) is at most
$$\frac{\varphi(\frac{n-1}{2})}{\frac{n-1}{2}}$$
where $\varphi$ is Euler's function. Conjecture \ref{conj} posits that this bound is asymptotically tight. 
\end{claim}

\begin{proof}
Recall that the number of elements of order $d$ in a cyclic group of order $m$, is $\varphi(d)$ for every divisor $d$ of $m$. Since the cyclic group $\mathbb{F}_n^*$ contains $n-1$ elements, it is left to prove that 
\begin{itemize}
    \item When $n\equiv 1\bmod 4$: 
    $$\frac{\varphi(\frac{n-1}{2})}{\frac{n-1}{2}}=\frac{\varphi(n-1)}{n-1}$$
    \item When $n\equiv 3\bmod 4$:
    $$\frac{\varphi(\frac{n-1}{2})}{\frac{n-1}{2}}=\frac{\varphi(n-1)+\varphi(\frac{n-1}{2})}{n-1}$$
\end{itemize}
These claims follow from the multiplicative properties of $\varphi$. Namely, 
\begin{equation}\label{eq:totient_gcd}
    \varphi(m\cdot k)=\frac{d\cdot\varphi(m)\cdot\varphi(k)}{\varphi(d)}
\end{equation}
for two positive integers $m,k$, where $d=\gcd(m,k)$. The claim follows by applying Equation \ref{eq:totient_gcd} with $m=2,k=\frac{n-1}{2}$. If $n\equiv 1\bmod 4$, then $\varphi(n-1)=2\cdot\varphi(\frac{n-1}{2})$ and when $n\equiv 3\bmod 4$, $\varphi(n-1)=\varphi(\frac{n-1}{2})$.
\end{proof}

In figure \ref{fig:acyclic_percentage} observe the good agreement with the function presented in Claim \ref{conjecture_claim}. For example, the acyclic percentage is close to $1$ for primes $n$ of the form $\frac{n-1}{2}=p$, where $p$ is prime. The lowest acyclic percentage is attained when $\frac{n-1}{2}=2\cdot3\cdot5\dots$. All of this is in agreement with our Conjecture \ref{conj}.

\subsection{More non-acyclic cases}

As the two last figures illustrate, Theorem \ref{singularity_property_T} does not capture all the cases in which $X_{n,c}$ is non-acyclic. Indeed, there are simple linear dependencies between the rows or columns of the matrices $\underline{S}(1)$ and $\underline{S}(-1)$ even when $c$ is a generator of $\mathbb{F}_n^*$ or has order $\frac{n-1}{2}$ when $n\equiv3\mod{4}$. A simple example of non-acyclic cases not captured by Theorem \ref{singularity_property_T} is when $$c^2+c-1\equiv0\mod{n}$$ 
To see why in these cases $X_{n,c}$ is non-acyclic, we pick $k=\frac{n-1}{2}$, and define $\Tilde{v}_{n,k}$ to be $v_{n,k}$ as in the proof of Theorem \ref{singularity_property_T}, just with the single change $\Tilde{v}_{n,k}[-c-1]=0$. It is a straightforward verification with the same proof that in these cases $\Tilde{v}_{n,k}\cdot\underline{S}(\omega_k)=\Vec{0}$ using the fact that $-c-1\equiv-\frac{1}{c}\mod{n}$.

\section{Full Matrices}\label{section:full_matrices}

In this section we sketch another possible approach to the construction of an infinite family of 2-dimensional hyperpaths. A certain extension of the boundary operator matrix plays a key role in these developments. 

\begin{definition}\label{def:full matrix definition}
For a prime $n$ and $c\in\{2,\dots,n-3\}$ the matrix $F=F_{n,c}$ is an $(n^2-1)\times (n^2-1)$ binary matrix. A row $\rho_{xy}$ of $F$ is indexed by an ordered pair $(x,y)\in \mathbb{F}_n^2\setminus(0,0)$. A column $f_{(xyz)}$ is indexed by an ordered triple $(x,y,z)$, where $(x,y)\in \mathbb{F}_n^2\setminus(0,0)$ and $z=-\frac{x+y}{c}$ i.e., $x+y+c\cdot z=0$. The $[(u,v),(x,y,z)]$ entry of $F$ is $1$ iff  
\begin{enumerate}
    \item $(u,v)=(x,y)$, this is called an $xy$-entry.
    
    \item $(u,v)=(y,z)$, this is called an $yz$-entry.
    
    \item $(u,v)=(z,x)$, this is called an $zx$-entry.
\end{enumerate}
\end{definition}

Note that the $xy$ entries in $F$ form a permutation matrix which we call $P_{xy}$. Likewise for $P_{yz}$ and $P_{zx}$. Also $P_{xy}=I$. Consequently:
\begin{equation}\label{eq:sum_of_3}
F=I+P_{yz}+P_{zx},   
\end{equation}
Where $P_{yz}, P_{zx}$ are permutation matrices.\\ 
We next construct a matrix $M_{n,c}$. It has one row $R_{u,v}=\rho_{u,v}-\rho_{v,u}$ for each ordered pair $u < v$.
A column $\mu_{(xyz)}$ of $M_{n,c}$ is indexed by a triple $(x,y,z)\in \mathbb{F}_n^3$, where $x+y+cz=0, ~x<y, ~x\neq z$ and $y\neq z$.

\begin{claim}\label{turning_full_matrix_into_A}
$M_{n,c}=A_{n,c}$.
\end{claim}
\begin{proof}
That $M_{n,c}$ and $A_{n,c}$ have the same dimensions follows from Claim \ref{right_number}. Moreover, the rows and columns of the two matrices are identically indexed. It remains to show that like $A_{n,c}$, the matrix $M_{n,c}$ is a minor of the boundary matrix, i.e.,
\begin{equation*}
\mu_{(xyz)}=e_{(x,y)}+e_{(y,z)}+e_{(z,x)}
\end{equation*}
where for $u<v$ we define $e_{(u,v)}$ to be the $0/1$ column vector with a single $1$ in position $(u,v)$. We extend the definition to the range $u>v$ via $e_{(u,v)}=-e_{(v,u)}$. Note that here we never need to deal with the case $u=v$. Since $x<y$ in both rows and columns indexing, the $(x,y)$ entry on both sides is $1$. Recall: (i) In the indexing of columns $z\neq x, y$, (ii) The $(y,z)$ entry of $f_{(xyz)}$ is $1$ and its $(z,y)$ entry is $0$, (iii) The $(z,x)$ entry of $f_{(xyz)}$ is $1$ and its $(x,z)$ entry is $0$, (iv) $R_{u,v}=\rho_{u,v}-\rho_{v,u}$. Therefore if $z>y$, the $(y,z)$ entry of $\mu_{(xyz)}$ is $1$, while if $y>z$, it is $-1$, as claimed. Likewise, if $z<x$, the $(z,x)$ entry of $\mu_{(xyz)}$ is $1$, while if $z>x$, it is $-1$, again as claimed.

\end{proof}

Note that $\text{rank}(F_{n,c})\le n^2-n$ since there are $n-1$ pairs of identical columns in $F_{n,c}$, namely, $f_{(a,-(c+1)a,a)}=f_{(-(c+1)a,a,a)}$ for every $a\in \mathbb{F}_n^*$.

\begin{claim}\label{claim:full-rank-to-full-rank}
If $\text{rank}(F_{n,c})=n^2-n$, then $\text{rank}(A_{n,c})=\binom{n-1}{2}$, in which case $X_{n,c}$ is a hypertree. 
\end{claim}

\begin{proof}
As mentioned earlier, removing from $F$ all $n-1$ columns $f_{(-(c+1)a,a,a)}$ over $a\in \mathbb{F}_n^*$ does not decrease the rank of $F$. Also, the $n-1$ rows $\{\rho_{0,k}\suchthat k\in \mathbb{F}_n^*\}$ are in the linear span of the other rows and can be removed from $F$ without reducing the rank. The linear dependence $$\rho_{0,k}=\sum_{j\neq k}\rho_{k,j}-\sum_{i\neq k,0}\rho_{i,k}$$ follows from 
\begin{equation}\label{eq:linear_dependence_star}
    \sum_{j\neq k} \rho_{k,j}=\sum_{i\neq k} \rho_{i,k}
\end{equation}
Equation (\ref{eq:linear_dependence_star}) states that two row vectors are equal. Let us check this for each coordinate. Since coordinates correspond to columns, let us consider the column indexed by $(u,v,w)$. If $k\not\in\{u,v,w\}$, the corresponding coordinate in both vectors is $0$. If $k\in\{u,v,w\}$, this coordinate in both vectors is $1$. (Note that no column is indexed by $(k,k,k)$ since $c\neq-2$). \\
In other words, $\text{rank}(\mathbf{F})=\text{rank}(F)$, where $\mathbf{F}=\mathbf{F}_{n,c}$ is the $(n^2-n)\times (n^2-n)$ matrix that results from $F=F_{n,c}$ after these column and row deletions. Also, as discussed in section \ref{from_A_to_MCB_S}, $\text{rank}(\mathbf{A})=\text{rank}(A)$, where $\mathbf{A}=\mathbf{A}_{n,c}$ the $\binom{n-1}{2}\times\binom{n-1}{2}$ matrix that is obtained from $A=A_{n,c}$ after deleting the set of $n-1$ rows $\{\rho_{0,k}\suchthat k\in \mathbb{F}_n^*\}$. Claim \ref{turning_full_matrix_into_A} describes a linear relation between $F$ and $A$. The same linear transformation captures also the relation between $\mathbf{F}$ and $\mathbf{A}$.

\begin{equation}\label{F_into_A_equation}
    \begin{pmatrix}
    I_{\binom{n-1}{2}} & -I_{\binom{n-1}{2}} & 0_L 
    \end{pmatrix}
    \mathbf{F} 
    \begin{pmatrix}
    I_{\binom{n-1}{2}} \\
    0_R
    \end{pmatrix}
    =\mathbf{A}
\end{equation}
Where $0_L$ is the all-zero matrix of size $\binom{n-1}{2}\times(2n-2)$ and $0_R$ is the all zero matrix of size $(\n^2-n-\binom{n-1}{2})\times\binom{n-1}{2}$. \\
It are left to prove that if $\mathbf{F}$ is invertible then so is $\mathbf{A}$. \\
$\mathbf{F}$'s rows are indexed as follows: the first $\binom{n-1}{2}$ rows correspond to pairs $(x,y)$ with $1\leq x<y\leq n-1$. The row $\binom{n-1}{2}$ below that of $(x,y)$ corresponds to $(y,x)$. Finally, $n-1$ rows for pairs $(x,x)$ and another $n-1$ for pairs $(x,0)$ over $x\in \mathbb{F}_n^*$.

The first $\binom{n-1}{2}$ columns in $\mathbf{F}$ are indexed by ordered triples $(x,y,z)$ with $x<y$ and $x+y+cz=0$. At $\binom{n-1}{2}$ positions to the right of column $(x,y,z)$ is column $(y,x,z)$. Then come $n-1$ columns $(x,x,-\frac{2x}{c})$ and another $n-1$ columns $(-(c+1)x,x,x)$ over every $x\in \mathbb{F}_n^*$. Define 

\begin{equation}\label{F_L_into_F_equation}
    \mathbf{F}_L:=\begin{pmatrix}
    I_{\binom{n-1}{2}} & -I_{\binom{n-1}{2}} & 0_L 
    \end{pmatrix}
    \mathbf{F}
\end{equation}
Since $\mathbf{F}$ is invertible, there holds: 
\begin{equation*}
    \text{rank}(\mathbf{F}_L)=
    \text{rank}(
    \begin{pmatrix}
    I_{\binom{n-1}{2}} & -I_{\binom{n-1}{2}} & 0_L 
    \end{pmatrix})=
    \binom{n-1}{2}
\end{equation*}

We turn to show that the first $\binom{n-1}{2}$ columns in $\mathbf{F}_L$ are linearly independent.
In light of equation \ref{F_into_A_equation} this implies that $\mathbf{A}$ has a full rank. 

\begin{enumerate}
\item\label{argument_1_zero_columns} A column indexed $(x,x,-\frac{2x}{c})$ in $\mathbf{F}_L$ is all-zero: Let us see which rows in $$\begin{pmatrix}
I_{\binom{n-1}{2}} & -I_{\binom{n-1}{2}} & 0_L 
\end{pmatrix}$$ have a non-zero entry in this column in $\mathbf{F}$. The only such row is the one indexed by $x\cdot(1,-\frac{2}{c})$. In Equation (\ref{F_L_into_F_equation}) we obtain $+1$ for $x\cdot(1,-\frac{2}{c})$ and $-1$ for $x\cdot(-\frac{2}{c},1)$, for a total of $0$.
\item As in item \ref{argument_1_zero_columns} every column indexed by $x\cdot(y,1,1)$ for $y=-1-c$ in $\mathbf{F}_L$ is all zeros.
\item Consequently the first $2\cdot\binom{n-1}{2}$ columns in $\mathbf{F}_L$ have rank $\binom{n-1}{2}$. We next see that the first $\binom{n-1}{2}$ columns are just next $\binom{n-1}{2}$ columns times $-1$, finishing the proof. In column $(x,y,z)$ of $\mathbf{F}$ there are $1$ entries in rows $(x,y),\ (y,z)$ and $(z,x)$ whereas column $(y,x,z)$ has $1$ entries in exactly the opposite edges, $(y,x),\ (z,y)$ and $(x,z)$. Therefore $$\mathbf{F}_L[:,(x,y,z)]=-\mathbf{F}_L[:,(y,x,z)]$$
\end{enumerate}
\end{proof}

\section*{Acknowledgement}
We thank Roy Meshulam for insightful comments on this manuscript.

\appendix
\section{Appendix: Matrix map}\label{appendix:_matrix_map}
It may not be very easy to keep track of the many matrices defined in this paper. In this section we attempt to review the connections between them and provide hyperlinks to the basic definitions.
\subsection{A map for Sections \ref{section:MCB}-\ref{section:Non-acyclic Complexes}}
In Section \ref{section:MCB} we introduce the class $MCB_{r,t}(\mathbb{Q})$ of $rt\times rt$ matrices with rational entries. Such a matrix $E$ is a $t\times t$ matrix of blocks each of which is an $r\times r$ circulant matrix (\ref{MCB_definition}). We express every circulant block as a polynomial in the permutation matrix $P$ (\ref{eq:permutation_matrix}), so that $E=E(P)$ (\ref{E(P)}) is a $t\times t$ matrix whose entries are polynomials of degree $\le r$ in $P$. Alternatively $E(P)$ is a $t\times t$ matrix over the polynomial ring $\mathcal{R}$ (\ref{eq:R_polynomial_ring}). We can evaluate each polynomial in $E(P)$ at a scalar $z\in\mathbb{C}$ to obtain a $t\times t$ complex matrix $\underline{E}(z)$ as in Equation (\ref{underscore(E)(z)}). The matrices $\mathcal{F}_r,\mathcal{L},Q,X$ appear in the proof of Theorem \ref{reduction_small_matrices}. In Section \ref{from_A_to_MCB_S} we find a rank-preserving transformation of $A$, the $\binom{n}{2}\times\binom{n-1}{2}$ boundary operator matrix of $X$, (see the background section \ref{background_simplicial_combinatorics}) to $S\in MCB_{n-1,\frac{n-3}{2}}(\mathbb{Q})$ (\ref{eq:S}).  

\subsection{A map for section \ref{section:full_matrices}}
Definition \ref{def:full matrix definition} introduces the 'full' matrix $F=F_{n,c}$, which is a sum of three permutations (\ref{eq:sum_of_3}). We transform $F_{n,c}$ into $M_{n,c}$ using simple row differences and deleting some column and rows. We observe (Claim \ref{turning_full_matrix_into_A}) that $M_{n,c}$ coincides with the boundary matrix $A_{n,c}$. The matrices $\mathbf{F}_{n,c}$ and $\mathbf{A}_{n,c}$ appear in Claim \ref{claim:full-rank-to-full-rank} which shows that if $F_{n,c}$ is of full rank then so is $A_{n,c}$.

\printbibliography
\end{document}